\documentclass[a4paper,12pt]{article}
\usepackage[top=2.5cm,bottom=2.5cm,left=3cm,right=3cm]{geometry}

\usepackage{float} 

\usepackage[final]{pdfpages}
\usepackage{graphics}
\usepackage{graphicx} 

\usepackage{epstopdf}
\usepackage{epsf,epsfig,latexsym}

\usepackage{multirow}
\usepackage{xcolor}
\usepackage{amssymb}
\usepackage{amsmath}
\usepackage{amsthm}
\usepackage{mathrsfs}
\usepackage{xspace}
\usepackage{complexity}
\usepackage{hyperref}
\usepackage{todonotes}
\usepackage{caption}
\usepackage{subcaption}
\usepackage[framemethod=tikz]{mdframed}
\usepackage{bibentry}

\usepackage{todonotes}

\newtheorem{definition}{Decfinition}
\newtheorem{lemma}{Lemma}
\newtheorem{theorem}{Theorem}
\newtheorem{proposition}{Proposition}
\newtheorem{corollary}{Corollary}
\newtheorem{conjecture}{Conjecture}
\newtheorem{problem}{Problem}
\newtheorem{claim}{Claim}
\newtheorem{remark}{Remark}
\newtheorem{observation}{Observation}

\newcommand \Ceil[1]
{
	\left	\lceil {#1}\right \rceil
}
\def \setg {{\cal G}}
\def \setf {{\cal F}}
\def \setx {\mathscr {X}}
\def \sety {\mathscr {Y}}

\colorlet{darkred}{red!60!black}
\colorlet{darkblue}{blue!60!black}
\usepackage{hyperref}
\hypersetup{
    colorlinks=true,
    citecolor=darkblue,     
    linkcolor=darkred,       
    urlcolor=green    
}

\usepackage{color}
\newcommand {\red} {\textcolor{red}}

\begin{document}

\title{\bf The minimum crossing number and minimum size of maximal 1-plane graphs with given connectivity\thanks{The work was supported by the National
Natural Science Foundation of China (Grant Nos. 12371346, 12371340, 12271157)}
}

\author{ Zhangdong Ouyang \\
{\footnotesize Department of Mathematics, Hunan First Normal University , Changsha 410205, P.R.China} \\
{\footnotesize E-mail:  oymath@163.com }\\
Yuanqiu Huang\thanks{Corresponding author.} \\
{\footnotesize  Department of Mathematics, Hunan Normal University, Changsha 410081, P.R.China} \\
{\footnotesize E-mail:  hyqq@hunnu.edu.cn}\\
Licheng Zhang\\
{\footnotesize School of Mathematics, Hunan University, Changsha 410082, P.R.China} \\
{\footnotesize E-mail: lczhangmath@163.com }\\
Fengming Dong\\
{\footnotesize National Institute of Education, Nanyang Technological University, 
Singapore} \\
{\footnotesize E-mail: fengming.dong@nie.edu.sg}\\
}

\date{}
\maketitle

\begin{abstract}
A $1$-planar graph is a graph which has a drawing on the plane
such that each edge is crossed at most once. 
If a $1$-planar graph is drawn in 
	that way,  the drawing is called a {\it $1$-plane graph}. 
A graph is maximal 1-plane (or 1-planar) if no additional edge can be added without violating 1-planarity or simplicity. It is known that any maximal 1-plane  graph is $k$-connected for some $k$ with $2\le k\le 7$. 
Recently, Huang et al. proved that any maximal 1-plane graph with 
$n$ ($\ge 5$)
vertices has at least $\Ceil{\frac{7}{3}n}-3$ edges, 
which is tight for all integers $n\ge 5$.  
In this paper, 
we study $k$-connected 
maximal 1-plane graphs for 
each $k$ with $3\le k\le 7$, 
and establish a lower bound for 
their crossing numbers
and a lower bound for their edge numbers, respectively.
\end{abstract}



\section{Introduction}\label{sec:introduction}

All graphs considered here are simple, finite,  and undirected unless otherwise stated. 
All terminology not defined here is referred to \cite{JAB}. 
For any graph $G$, let $V(G)$ and $E(G)$ denote its vertex set and edge set, respectively, and 
the {\it order}
and {\it size} of $G$
are defined to be $|V(G)|$ and $|E(G)|$,
respectively. 
For any $v\in V(G)$, 
let $N_G(v)$ denote the set of neighbors of $v$ in $G$,
and let $deg_G(v)=|N_G(v)|$. 
A {\it drawing} of a graph $G$ is a mapping $D$ that assigns to each vertex in $V(G)$ a distinct point in the plane and to each edge $uv$ in $E(G)$ a continuous arc connecting $D(u)$ and $D(v)$. We often make no distinction between a graph-theoretical object (such as a vertex, or an edge) and its drawing.
All drawings considered here are {\em good} unless otherwise specified, meaning that no edge crosses itself, no two edges cross each other more than once, and no two edges incident with the same vertex cross each other.  We denote by $cr_D(G)$
the number of crossings in the drawing $D$ of a graph $G$. The {\it
crossing number} $cr(G)$ of a graph $G$ is defined as the minimum number of
crossings in any drawing of $G$, and the corresponding drawing is
called an {\it optimal} drawing. It is evident that an optimal
drawing is always a {\it good} drawing, which means that no edge crosses
itself, no two edges cross more than once, and no two edges incident
with the same vertex cross each other. Unless otherwise specified, all drawings considered in this paper are assumed to be good. For further information on the crossing number of graphs,  we refer to \cite{MS}.

A drawing $D$ of a graph 
is $1$-{\it planar} if each edge in $D$ is crossed at most once. 
If a graph has a $1$-planar drawing, then it is called 1-{\it planar}. A graph together with a $1$-planar drawing is called a {\it $1$-plane graph}. To avoid confusion, in this paper, we use $cr_{\times}(G)$ to denote  the number of crossings in the corresponding 1-planar drawing of the 1-plane graph $G$.

The notion of 1-planar graphs was first introduced in 1965  by Ringel \cite{GR} in connection with the problem of simultaneous coloring of the vertices and faces of plane graphs. Since then many properties of $1$-planar graphs have been studied (e.g. see the survey paper \cite{SK}). It is known that any $1$-planar graph with $n$ vertices has at most $4n-8$ edges \cite{IF,JP}, and this bound is tight for $n=8$ and $n\ge 10$. A $1$-planar graph with $n$ vertices and $4n-8$ edges is called {\it optimal}. A $1$-planar graph  is {\it maximal} if adding any edge to it yields a graph which is not $1$-planar or not simple. A $1$-planar drawing is {\em maximal} if no further edge can be added to it such that the resulting drawing is still $1$-planar. Clearly, a graph $G$ is maximal $1$-planar if and only if every $1$-planar drawing of $G$ is maximal. A maximal 1-plane graph $G$ is called {\it immovable} if $cr_{\times}(G)\le cr_{\times}(G')$ holds for any 1-plane graph $G'$
which is obtained from $G$ by redrawing exactly one edge of $G$. 
For any integer $k$ with $3\le k\le 7$, 
let $\setg_k$ denote the set of $k$-connected maximal 1-plane graphs $G$,  where $G$ is required to be immovable if $k=3$.

In this article, we will 
establish a sharp lower bound of $cr(G)$ 
and a sharp low bound of $|E(G)|$ 
over all graphs $G\in \setg_k$, where $3\le k\le 7$.
In the remainder of this section, 
we first introduce some 
known results 
on $cr(G)$  and $E(G)$
for maximum $1$-plane graphs $G$,
and then present our main results
in this article.

\subsection{Known results}

The crossing number is a crucial parameter for assessing the 1-planarity of a graph or verifying the maximality of a 1-planar graph. It is well-known that any 1-plane graph with $n$ vertices admits at most $n-2$ crossings~\cite{JCD,YS}. For maximal 1-planar graphs, a tighter upper bound has been established~\cite{OHD}: if $G$  is a maximal 1-planar graph with $n$ vertices, then $cr(G) \le n-2 - (2\lambda_1 + 2\lambda_2 + \lambda_3)/6$, 
where $\lambda_1$ and $\lambda_2 $ denote the number of 2-degree and 4-degree vertices, respectively, and $\lambda_3$  counts the odd-degree vertices $w$ such that 
either $d_G(w) \le 9$ or $G-w$ is 2-connected.
In the case of optimal 1-planar graphs, the crossing number attains the maximum value $cr(G) = n-2$~\cite{OGC}.  Despite these upper bounds, lower bounds on the crossing numbers of 1-plane graphs remain largely unexplored. In this work, we investigate the lower bounds of crossing numbers for maximal 1-plane graphs, and obtain the following Theorem~\ref{main1}.

It is well-known that every maximal planar graph with $n$ vertices has exactly $3n-6$ edges.  
However,  this property 
cannot be extended to 1-planar graphs.
Surprisingly, there exist maximal 1-planar graphs that are even sparser than maximal planar graphs with the same order (see \cite{JB,FJ,PE,DTY1}).  This raises a question: how sparse can they possibly be?   This extremal problem for the minimum size of $1$-plane (or 1-planar)  graphs  has attracted much interest  (see \cite{DWGL} and the references therein for the definitions and numerous results).

Brandenburg et al. \cite{FJ} were the first to construct a class of maximal
1-plane graphs with $n$ vertices having only $\frac{7}{3}n+O(1)$ edges, and obtained the
following result.
\begin{theorem}[\cite{FJ}]\label{TA}
For any maximal 1-plane graph $G$ with $n\ge 4$ vertices, $|E(G)|\ge \frac{21}{10}n-\frac{10}{3}$. 
\end{theorem}

This bound was later improved by  Barát and Tóth \cite{JB},  who derived a tighter lower bound.
\begin{theorem}[\cite{JB}]\label{TB}
For any maximal 1-plane graph $G$ with $n\ge 4$ vertices, $|E(G)|\ge \frac{20}{9}n-\frac{10}{3}$. 
\end{theorem}

In \cite{HOZD},  we settled the optimal lower bound, proving the following tight result.
\begin{theorem}[\cite{HOZD}]\label{TC}
For any maximal 1-plane graph $G$ with $n\ge 5$ vertices, $|E(G)|\ge \frac{7}{3}n-3$,  and 
for every integer $n\ge 5$, 
there exists a maximal 1-plane graph $G$ such that $|E(G)|= \frac{7}{3}n-3$.
\end{theorem}

\subsection{Main results}

Our first main result is on the lower bounds of $cr(G)$ for $G\in \setg_k$,
where $3\le k\le 7$.

\begin{theorem}\label{main1}
	Let $G\in \mathcal{G}_k$ and $n=|V(G)|$,
		where $3\le k\le 7$.
	Then
	\begin{equation*}
		cr(G)\ \ge\  \left\{
		\begin{aligned}
			&\frac{n-2}{3},  \qquad &if\, k&=3  \,\,\text{and}\,\, n\ge 5, \\
			&\frac{n-2}{2},  &if\, k&=4  \,\,\text{and}\,\, n\ge 6, \\
			&\frac{3n-6}{5}, &if\, k&=5,6,\\
			&\frac{3n}{4}, &if\, k&=7.
		\end{aligned}
		\right.
	\end{equation*}
	Furthermore, when $k\in \{3,4,6\}$, the above lower bounds
	are tight for infinitely many integers $n$, and when $k=7$, 
	the lower bound is tight for $n\in \{24, 56\}$.
\end{theorem}

Our second main result is 
on the lower bound of $|E(G)|$
for $G\in \setg_k$,
where $3\le k\le 7$.
Note that the extremal graphs in Theorem \ref{TC} (see \cite{HOZD})
have connectivity $2$ when $n>5$. 
Intuitively, as the connectivity increases, the size of maximal 1-plane graphs will also increase. Any 1-plane  graph is $k$-connected
for some integer $k$ with $2\le k\le 7$
(see \cite{IF,YS}).  In this paper, we establish the lower bounds for the size of maximal 1-plane graphs with given connectivity, as stated below.

\begin{theorem}\label{main2}
	Let $G\in \mathcal{G}_k$
	and $n=|V(G)|$, where $3\le k\le 7$.
	Then
	\begin{equation*}
		|E(G)|\ge\left\{
		\begin{aligned}
			&\frac{10}{3}(n-2),  &if\, k&=3  \,\,\text{and}\,\, n\ge 5, \\
			&\frac{7}{2}(n-2),  &if\, k&=4  \,\,\text{and}\,\, n\ge 6, \\
			&\frac{18}{5}(n-2),  &if\, k&=5,6,\\
			&\frac{15}{4}(n-2) +\frac{3}{2},
			\quad  &if\, k&=7.
		\end{aligned}
		\right.
	\end{equation*}
	Furthermore, when $k\in \{3,4,6\}$, the above lower bounds
are tight for infinitely many integers $n$, and when $k=7$, 
the lower bound is tight for 
$n\in \{24, 56\}$.
\end{theorem}
 
In order to display the relevant information more intuitively, the known results on lower bounds for the crossing number and the size of maximal
1-plane graphs $G$ with given connectivity $k$ are collected in Table~\ref{tab1}. 

\captionsetup[table]{position=bottom}
\begin{table}[ht]
    \centering  
    \setlength{\tabcolsep}{12pt}
    \newcommand{\spc}{\hspace{3pt}}
      \renewcommand{\arraystretch}{1.2} 
\begin{tabular}{c|c|c|c|c|c}
 &$k=2$ & $k=3$ & $k=4$& $k=5,6$& $k=7$\\   \hline
$cr(G)\ge $     &?&$\frac{1}{3}n-\frac{2}{3}$&$\frac{1}{2}n-1$ &$\frac{3}{5}n-\frac{6}{5}$ &$\frac{3}{4}n$ \\
$|E(G)|\ge $      & $\Ceil{\frac{7}{3}n}-3$
\cite{HOZD}&$\frac{10}{3}n-\frac{20}{3}$&$\frac{7}{2}n-7$&$\frac{18}{5}n-\frac{36}{5}$ &$\frac{15}{4}n-6$\\
\end{tabular}
    \caption{The known lower bound for the crossing number and size of $k$-connected maximal 1-plane  graphs. Note: for $k=3$, the maximal 1-plane  graphs are required to be immovable.}
    \label{tab1}
\end{table}

The rest of this paper is structured as follows. In Section~\ref{sec-pre}, we present some elementary results on maximal 1-plane graphs.
In Section~\ref{sec-exmg}, we construct maximal $1$-plane graphs in $\setg_k$
for each $3\le k\le 7$
that demonstrate the tightness of the lower bound in our main results.
In Section~\ref{sec-main}, we complete the proofs of Theorems~\ref{main1} and~\ref{main2}.
Finally, in Section~\ref{sec:problem}, we propose open problems for further research.

\section{Preliminaries}\label{sec-pre}

A planar drawing partitions the plane into connected regions called the {\it faces}. Each face is bounded by a closed walk (not necessarily a cycle) called its {\it boundary}.
Two faces $F_1$ and $F_2$ are said to be {\it adjacent} if their boundaries share at least one common edge. By $\partial(F)$ we denote the set of vertices on the boundary of face $F$. A face $F$ is called a {\it triangle} if $|\partial(F)|=3$. A
{\it triangulation} (also known as maximal plane graph) is a plane graph 
in which all faces are triangles.  
The {\it dual graph}  $G^*$ of a plane graph $G$ is a graph that has a vertex corresponding to each face of $G$, 
and an edge joining each pair of vertices in $G^*$ which correspond to two adjacent faces  in $G$.  A {\it cut-set} of a  connected 
graph $G$ is a subset $S$ of $V(G)$ such that 
$G-S$  has more than one 
component.
A {\it minimum cut-set} of a graph is a cut-set of smallest possible size. A {\it minimal cut-set} is an cut-set of a graph that is not a proper subset of any other cut-set. Every minimum cut-set is a minimal cut-set, but the converse does not necessarily hold.

For any 1-plane graph $G$, let $G^\times$ be the plane graph obtained by replacing each crossing with a vertex of degree 4, $G_P$ be the plane graph obtained by  removing one edge from each crossing pair in $G$, and $G_P^*$ be the dual of $G_P$. 
Clearly, $G_p$ depends on the edges
that are removed, and is not unique. 
A vertex in  $G^\times$ is called {\it fake} if it corresponds to some crossing of $G$, and is {\it  true} otherwise.  A face of $G^\times$ is called {\it fake} if it is incident with some fake vertex in $G^\times$, and is {\it true} otherwise.
Each face in $G_P$ is either a true face 
of $G^\times$, called a {\it blue face}
of $G_P$, 
or can be obtained by 
merging at least two adjacent fake faces of $G^\times$, 
called a {\it red face} of $G_P$. 

A vertex in $G_P^*$ is termed 
either {\it red} or  {\it blue},
depending  on whether its corresponding face in $G_P$ is red or blue, respectively.
An edge $e$ of $G$ is called {\it non-crossing} if 
it does not cross other edges in $G$ and is {\it crossing} otherwise.

Let $\setf_{fk}(G^\times)$ and  $\setf_{tr}(G^\times)$ be the set of all fake faces and all true faces of $G^\times$, respectively, $\setf_{rd}(G_P)$ and  $\setf_{bl}(G_P)$ the set of red faces and blue faces of $G_P$, respectively, and $V_{rd}(G_P^*)$ and $V_{bl}(G_P^*)$  the set of red vertices and blue vertices of $G_P^*$, respectively. 
The following two observations 
are obtained directly by definition.

\begin{observation}\label{obs}
	For any maximal 1-plane graph $G$
	of order $n$ $(\ge 3)$,
	each red vertex  in $G_P^*$ 
	is adjacent to some other 
	red vertices in $G_P^*$, and 
\begin{equation*}
		|V_{rd}(G_P^*)|=|\setf_{rd}(G_P)|\le \frac{1}{2}|\setf_{fk}(G^\times)|, 
		\quad 
		|V_{bl}(G_P^*)|
		=|\setf_{bl}(G_P)|=|\setf_{tr}(G^\times)|.
\end{equation*}
\end{observation}

\begin{observation}\label{tri}
For any maximal 1-plane graph $G$
	of order $n$ $(\ge 3)$,
if $G^\times$ is a triangulation, then the following hold:
\begin{enumerate}
\setlength{\itemindent}{2em}
\item [{\rm(i)}] $G_P$ is also a triangulation with $3n-6$ edges and $2n-4$ faces; 
\item [{\rm(ii)}] $G_P^*$ is a 3-regular  plane graph with $2n-4$ vertices and  $3n-6$ edges; and 
\item [{\rm(iii)}] $cr_{\times}(G)=\frac{1}{2}|V_{rd}(G_P^*)|$ and $|E(G)|=3n-6+cr_{\times}(G)$.
\end{enumerate}
\end{observation}

Three known 
properties on maximal 1-plane graphs are given below
(see \cite{JB,FJ} and \cite{OHD}).

\begin{lemma}[\cite{JB,FJ}]
	\label{adjacent} 
For any face $F$ of a maximal 1-plane graph $G$, 
$\partial(F)$ contains at least two vertices; and any two true vertices   in 
$\partial(F)$ are adjacent in $G$.
\end{lemma}

For any non-empty subset $S\subseteq V(G)$,
let $G[S]$ denote the subgraph of $G$
induced by $S$.

\begin{lemma}[\cite{JB,FJ}]\label{K_4} 
For any two edges $ab$ and $cd$
	in a maximal 1-plane graph $G$,
	if they cross each other, 
then $G[\{a,b,c,d\}]\cong K_4$.
\end{lemma}


\begin{lemma}[\cite{OHD}]\label{lemtr-3}
Let $G$ be a $3$-connected maximal 1-plane graph. If $G$ is immovable, then $G^\times$ is a triangulation.
\end{lemma}

In the following, we show that 
the conclusion of Lemma~\ref{lemtr-3} 
also holds if the condition that 
$G$ is immovable is replaced by 
that $G$ is $4$-connected. 

\begin{lemma}\label{lemtr-k}
Let $G$ be a $4$-connected maximal 1-plane graph.  Then $G^\times$ is a triangulation.
\end{lemma}

\begin{proof}
Suppose that $G^\times$ is not a triangulation. Then $G^\times$ has a face $F$ bounded by a facial
cycle $C$ with at least four vertices. As any two fake vertices in $G^\times$ are not adjacent, $G^\times$
has two true vertices $u$ and $v$ in $C$ which are not adjacent in $C$. By Lemma~\ref{adjacent},   $u$ and $v$ must be adjacent in $G$. 
Observe that the edge $e=uv$ in $G$
joining $u$ and $v$ is not on $C$,
and $C$ can be divided 
into two paths with ends $u$ and $v$, say $P_1$ and $P_2$. 
Thus, each $P_i$ contains an internal vertex $z_i$, as shown in Figure~\ref{fig1}.  
We can draw a line segment $L_{uv}$ within face $F$  connecting $u$ and $v$. Then the Jordan closed curve $\mathcal{O}$, formed by $L_{uv}$ and $uv$, divides $G$ into two parts $G_1$ and $G_2$.  As each fake vertex in $G^\times$ is adjacent to four true vertices, regardless of whether  $z_1$ and $z_2$  are true vertices or not, both $G_1$ and $G_2$ must contain true vertices.

Assume that $e$ is a non-crossing edge in $G$. Then $G-\{u,v\}$ is disconnected. 
Assume that $e$ is crossed by some edge $e'=u'v'$ in $G$. 
Then one can verify that $G-\{u,v,u'\}$ or $G-\{u,v,v'\}$ is disconnected.
Both cases 	contradict the given condition that 
$G$ is $4$-connected. 

Thus, the result holds.
\end{proof}

\begin{figure}[H]
\centering
\includegraphics[width=6 cm] 
{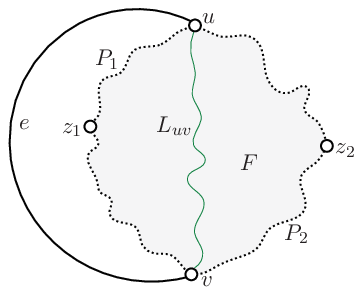}

\caption{A face $F$ of $G^\times$ bounded by a facial
cycle $C$ with at least four vertices}
\label{fig1} 
\end{figure}

\begin{lemma}\label{optimal}
Let $G$ be  a 1-plane graph.  If $G^\times$ is a triangulation, then $cr(G)=cr_{\times}(G)$.  
\end{lemma}

\begin{proof}
It is known that every planar graph with $n$ vertices has at most $3n-6$ edges. This implies that $cr(G)\ge |E(G)|-3|V(G)|+6$.	
Let $c=cr_\times(G)$. Note that $cr(G)\le c$. It suffices to show that $c=|E(G)|-3|V(G)|+6$.

Note that
$|V(G^\times)|=|V(G)|+c$ and $|E(G^\times)|=|E(G)|+2c$. 
Let $f$ be the number of faces in $G^\times$. Then, 
as $G^\times$ is a triangulation, 
$3f=2(|E(G)|+2c)$. 
Applying the Euler's formula
to $G^\times$ yields that 
\begin{equation*}
|V(G)|+c-(|E(G)|+2c)+\frac{2}{3}(|E(G)|+2c)=2,
\end{equation*}
implying that $c=|E(G)|-3|V(G)|+6$. 

This completes the proof.
\end{proof}

\begin{remark}\label{1pcross}
From Lemma~\ref{optimal},  we find that Observation \ref{tri} (iii) can be stated as  $cr(G)=\frac{1}{2}|V_{rd}(G_P^*)|$ and $|E(G)|=3n-6+cr(G)$.
\end{remark}

\begin{lemma}\label{join}
Let $G$ be a maximal 1-plane graph, and let  $F_1$ and $F_2$  be two faces of $G$ sharing a non-crossing edge $e$.
Then, for any two vertices 
$v_1$ and $v_2$ in $G$, 
if $v_i \in \partial(F_i) \setminus \partial(F_{3-i}) $ for each $i=1,2$, 
they are adjacent in $G$.
\end{lemma}

\begin{proof}
 Suppose, to the contrary, that $v_1$ and $v_2$ are non-adjacent in $G$. As the common edge $e$ of $F_1$ and $F_2$ is non-crossing, we can always add a new edge joining $v_1$ and $v_2$ such that it traverses across $F_1$ and $F_2$ and crosses $e$ exactly once.  This contradicts the maximality of $G$.
\end{proof}

\begin{lemma}\label{trface}
Let $G\in \mathcal{G}_k$ and $n=|V(G)|\ge 5$, 
	where $3\le k\le 5$.
Then,  
any true face $\Delta$ of $ \setf_{tr}(G^\times)$
	is adjacent to at most  $5-k$ 
	 true faces of $G^\times$, 
	where $n\ge 6$ when $k\ge 4$.
\end{lemma}

\begin{proof}
By Lemmas~\ref{lemtr-3} and~\ref{lemtr-k}, we know that $G^\times$ is a triangulation, 
implying that 
for each face $F$ of $G^\times$,
$|\partial(F)|=3$ and $F$ 
is adjacent to precisely three other faces of $G^\times$. 
The following claim follows directly.

\setcounter{claim}{0}

\begin{claim}\label{clfake}
{\it Assume that $c$ is a crossing point between edges $v_0v_2$ and $v_1v_3$ in $G$.  Then, in $G^\times$, $c$ is enclosed by four fake faces bounded by the closed curves $v_icv_{i+1}v_i$ , where $i=1,2,3,4$ and the indices are taken modulo 4.}
\end{claim}

Assume that $\partial(\Delta)=\{u,v,w\}$.

 (i).   Assume that $k=3$. Suppose, to the contrary, that $\Delta$ is adjacent to three true faces $\Delta_1$, $\Delta_2$
and $\Delta_3$ of $G^\times$. Note that two adjacent faces share two vertices and one edge in $G^\times$. Hence, we can assume that $\partial(\Delta_1)=\{u,v,x\},  \partial(\Delta_2)=\{u,w,y\}$,  and $\partial(\Delta_3)=\{v,w,z\}$. Without loss of generality, assume that $x,y,z$ lie outside $\Delta$. We first claim that $x=y=z$ cannot happen; 
otherwise, it can be easily verified that $G$ is a complete graph $K_4$ formed by $u,v,w$ and $x(=y,z)$,  
contradicting the condition 
$n\ge 5$.  
Without loss of generality,  assume that $x\neq y$. By Lemma~\ref{join},  it follows that $xw, yv\in E(G)$. As $\Delta$, $\Delta_1$ and  $\Delta_2$ all are true faces, edges $xw$ and $yv$ must lie outside region bounded by the closed curve $xuywvx$, forcing them to cross at a point $\alpha$ (as illustrated in Figure~\ref{fig2} (I)).  By Claim~\ref{clfake}, the closed curve $v\alpha wv$ bounds a fake face of $G^\times$. Clearly, this fake face is adjacent to $\Delta$, contradicting the assumption that $\Delta$ is adjacent only to true faces.

(ii).  Assume that $k=4$ and $n\ge 6$.
Suppose, to the contrary, that $\Delta$  is adjacent to two true faces $\Delta_1$ and $\Delta_2$ of $G^\times$. 
Likewise,  we can assume that $\partial(\Delta_1)=\{u,v,x\}$ and $\partial(\Delta_2)=\{u,w,y\}$, where $x,y$ lie outside $\Delta$. 
If $x=y$, then $deg_G(u)=3$, contradicting 
the condition that $k=4$.  
Thus, $x\neq y$.  
Lemma~\ref{join} implies that $xw,yv\in E(G)$.  Furthermore,  edges $xw$ and $yv$ must cross at a point $\alpha$ outside the region bounded by the closed curve $xuywvx$, see Figure~\ref{fig2} (I). By Claim~\ref{clfake}, the closed curves $x\alpha vx$,  $v\alpha wv$ and $y\alpha wy$ bound three fake faces of $G^\times$, respectively. 
Since $n\ge 6$, 
$G$ contains  vertices outside the region bounded by the closed curve $xuy\alpha x$, 
implying that $\{x,u,y\}$ 
is a cut-set of $G$,  
which contradicts the condition that
$G$ is $k$-connected, where $k=4$.

(iii). Assume that $k\ge 5$.  
Suppose, to the contrary, that $\Delta$ is  adjacent to one true face $\Delta_1$ of $G^\times$. 
Similarly, assume that $\partial(\Delta_1)=\{u,v,x\}$, where $x$ lies outside $\Delta$. Lemma~\ref{join} implies that $xw\in E(G)$ and $xw$ must lie outside the region bounded by the closed curve $xuwvx$, see Figure~\ref{fig2} (II).  
We first claim that $xw$ must be a crossing edge in $G$. 
Otherwise, $n\ge k+1\ge 6$
implies that either 
$\{x,v,w\}$ or $\{x,u,w\}$ 
is a cut-set of $G$, 
depending on whether the interior of the region bounded by the closed curve $xuwx$ contains no vertices or vertices, respectively.
 Assume that $xw$  crosses edge $tt'$  at point $\alpha$ in $G$, see Figure~\ref{fig2} (III).  
We now claim that $t\neq u$.  
Otherwise, by Claim~\ref{clfake}, the closed curves $u\alpha wu$ and $u\alpha xu$ bound two fake faces of $G^\times$, which implies that $deg_G(u)=4$, contradicting the condition that 
$G$ is $k$-connected, where $k\ge 5$.

Two cases remain to be considered depending on whether  $t'=v$ or $t'\neq v$.

If  $t'=v$, by Claim~\ref{clfake}, the closed curves $x\alpha vx$ and $w\alpha vw$ bound two fake faces of $G^\times$, which implies that $deg_G(v)=4$
and thus $N_G(v)$ is a cut-set of $G$.
If $t'\neq v$, then $\{x,u,w,t'\}$ forms a cut-set of $G$.
Both cases  contradict 
the condition that 
$G$ is $k$-connected, where $k\ge 5$.
\end{proof}

\begin{figure}[htp]
\centering
\includegraphics[width=0.9\textwidth]{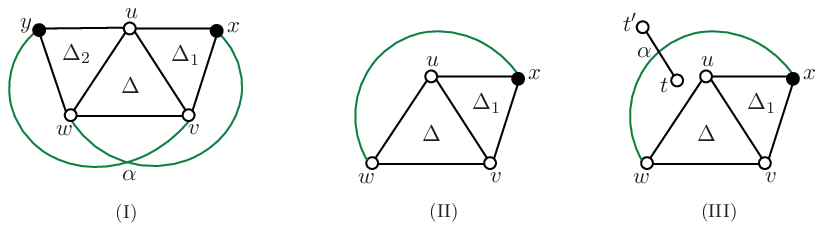}
\caption{Auxiliary graphs for proving Lemma~\ref{trface}}
\label{fig2} 
\end{figure}

\setcounter{corollary}{9}
\begin{corollary}\label{adjacent-vertex}
Let $G\in \mathcal{G}_k$ and $n=|V(G)|\ge 5$,
where $3\le k\le 5$. 
Then,  each blue vertex in $G^*_P$ 
is adjacent to at most  $5-k$ blue vertices in $G^*_P$, 
where $n\ge 6$ when $k\ge 4$.
\end{corollary}

\begin{proof}
Observe that the blue vertices of $G_P^*$ correspond one-to-one to the true faces of $G^\times$. Furthermore, two blue vertices of $G_P^*$ are adjacent if and only if  the true faces corresponding to them are adjacent in $G^\times$. Therefore, this is a direct corollary of Lemma~\ref{trface}.
\end{proof}

A drawing of a graph actually implies a rotation system. 
The rotation at a vertex is 
an order list of its incident edges 
in the clockwise direction.
In a drawing, two edges incident with some
vertex $w$ are said to be {\em consecutive} if they appear in sequence in the cyclic
ordering at $w$.  
Let $c_G(v)$ denote the number of crossing edges incident with $v\in V(G)$ in a 1-plane graph $G$. 
The following result provides bounds for  $c_G(v)$.

\begin{lemma}\label{trcredge}
Let $G\in \setg_k$, 
where $k\ge 5$. 
Then, for each vertex $v\in V(G)$, 
\begin{equation*}\label{le9-e1}
\left\lceil
\frac{deg_G(v)}{3}
\right\rceil
\le  c_G(v)
\le \left\lfloor\frac{deg_G(v)}{2}\right\rfloor.
\end{equation*}
\end{lemma}

\begin{proof}
By Lemmas~\ref{lemtr-k} and \ref{trface},  we know that $G^\times$ is a triangulation and there do not exist two adjacent true faces in $G^\times$.  Let $v\in V(G)$  with $deg_G(v)=l$ and $c_G(v)=s$. Since any two consecutive non-crossing edges of $G$ incident with $v$ must be on the boundary of the same true face of $G^\times$, we claim that no three consecutive non-crossing edges of $G$ incident with $v$.
Otherwise, the existence of two adjacent true faces in $G^\times$ would contradict Lemma~\ref{trface}. 

Clearly, there are $l-s$ non-crossing edges of $G$ incident with $v$. One can easily observe that the $s$ crossing edges of $G$ incident with $v$ can divide $l-s$ non-crossing edges of $G$ into at most $s$ parts, and one of the parts contains at least $\lceil\frac{l-s}{s}\rceil$ non-crossing edges. These non-crossing edges in the same part must be consecutive in $G$.  Thus, $\lceil\frac{l-s}{s}\rceil\le 2$. This implies that $s\ge \lceil\frac{l}{3}\rceil$.

We now claim that there are no two consecutive crossing edges of $G$ incident with $v$.
Otherwise, these two crossing points on the two consecutive crossing edges would appear on the boundary of a face of size at least 4 in $G^\times$, violating the triangulation property. Thus, $v$ is incident with at most $\lfloor\frac{l}{2}\rfloor$ crossing edges of $G$, i.e., $s\le \lfloor\frac{l}{2}\rfloor$.
\end{proof}

A cycle $C$ in a plane graph $G$ is called a {\it separating cycle} if both its interior and exterior contain vertices. A graph $G$ is called {\it $(a,b)$-regular} if the degree of each vertex in $G$ is either $a$ or $b$.
In this article, we will apply 
	the following results on a triangulation
due to Baybars \cite{IB}, Etourneau~\cite{EE}, Hakimi and Schmeichel~\cite{SLH}.

\begin{lemma}[\cite{IB,SLH}]\label{connectivity}
Let $G$ be a triangulation and let $S\subset V(G)$ be a minimal cut-set
of $G$. Then $S$ induces a separating cycle in $G$.
\end{lemma}

\begin{lemma}[\cite{EE}]\label{5connectivity}
If $G$ be a triangulation and is  (5,6)-regular, then $G$ is 5-connected.
\end{lemma}

\begin{lemma}[\cite{SLH}]\label{4connectivity}
Let $G$ be a triangulation with vertex degree sequence $d_1\ge d_2\ge \cdots
\ge d_p$,  where $d_p\ge 4$. If 
$  \left\lfloor\frac{7}{3}\omega(4)\right \rfloor+\omega(5)<14$,
then $G$ is $d_p$-connected, 
where $\omega(k)$ denote 
the number of integers $i$'s with $1\le i\le p$ such that $d_i=k$.
\end{lemma}


\section{Construction of extremal graphs}
\label{sec-exmg}

In this section, we construct 
some graphs in $\setg_k$
for the purpose of 
showing the sharpness of 
the main results in this article. 

\begin{definition}
Let $F$ be a face of a plane graph $G$.   Three operations of inserting
new vertices or edges within $F$ are introduced below:
\begin{itemize}
  \item {\it $K_1$-triangulation} for
  $|\partial(F)|\ge 3$: 
  insert a new vertex $x$ inside $F$ 
  and add new edges joining $x$ to all vertices in $\partial(F)$, 
  as shown in Figure~\ref{fig30} (I);
  
  \item {\it $K_2$-triangulation}
  for $|\partial(F)|=4$: 
  insert a complete graph $K_2$ inside $F$ and add six new edges joining 
  the two vertices of $K_2$ to the four vertices in $\partial(F)$ so that $F$ is triangulated without producing any crossing,
  as shown in Figure~\ref{fig30} (II);  and 
  
  \item {\it $T_\times$-triangulation}  for $|\partial(F)|=4$: 
  insert a pair of crossing edges
  within $F$ connecting the two pair of diagonal vertices in $\partial(F)$, respectively, as shown in Figure~\ref{fig30} (III).
\end{itemize}
\end{definition}

\begin{figure}[htp]
\centering
\includegraphics[width=0.9\textwidth]{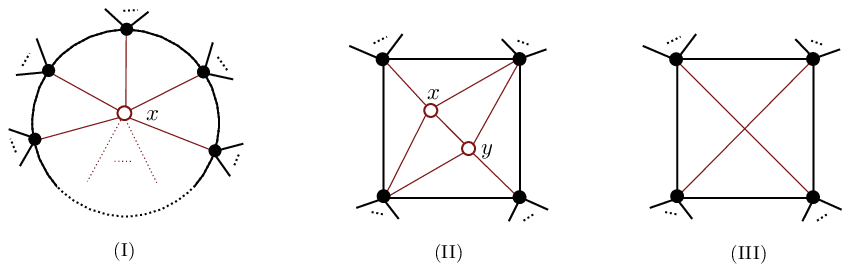}
\caption{Three triangulation operations}
\label{fig30}
\end{figure}

For $k\ge 1$, let $C_{2^{k+1}}$ be a cycle of length $2^{k+1}$ with $V(C_{2^{k+1}})=\{a^{k}_{i}|i=1,2,\cdots,2^{k+1}\}$ and $E(C_{2^{k+1}})=\{a^{k}_{i}a^{k}_{i+1}|i=1,2,\cdots,2^{k+1}\}$,
where indices are taken modulo $2^{k+1}$. 
Now we construct a plane graph $H^k$
obtained from $k$ cycles $C_{2^2},C_{2^{3}},\cdots,C_{2^{k+1}}$ ($C_{2^2}$ and $C_{2^{k+1}}$ are called {\it kernel} and {\it periphery} cycles, respectively)  as follows.  

For $k=1$, let $H^1$ be $C_4$,
  and for $k\ge 2$, let $H^k$ be the graph with  
\begin{equation*}\label{H(VK)}
V(H^k)=\bigcup_{i=1}^k V(C_{2^{i+1}}),
\end{equation*}
and 
\begin{equation*}\label{H(EK)}
E(H^k)=	
\bigcup_{i=1}^k E(C_{2^{i+1}})
	\cup 
	\bigcup_{j=1}^{k-1} 
	\bigcup_{i=1}^{2^{j+1}}\left\{ a^j_{i}a^{j+1}_{2i-2}, a^j_{i}a^{j+1}_{2i}\right\}.
\end{equation*}
where $a^j_0=a^j_j$.  The instance for $k=3$ can be seen on the left side of Figure~\ref{fig3}.

\begin{figure}[htp]
\centering
\includegraphics[width=0.9\textwidth]{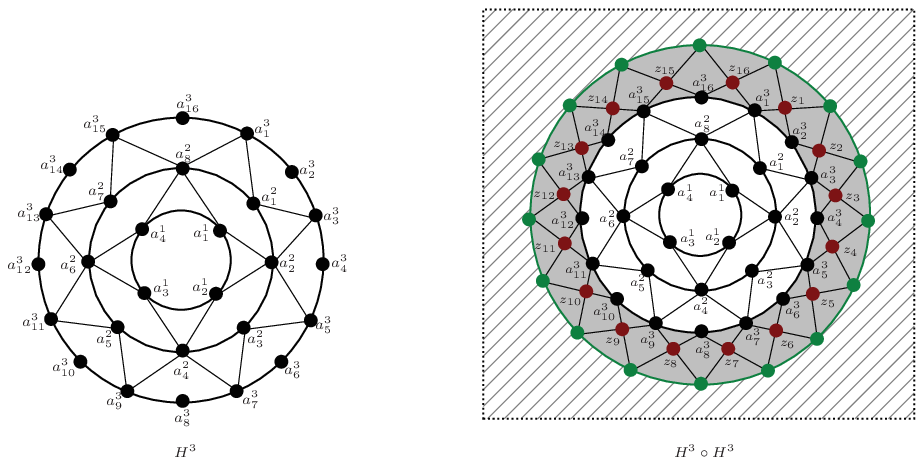}

\caption{Plane graphs $H^3$ and $H^3 \circ H^3$, where the structure outside the green circle $C_{2^{4}}$ is omitted}
\label{fig3} 
\end{figure}

We further construct a new plane graph $H^k \circ H^k$ from $H^k$ in the following way: 
First, perform a stereographic projection of $H^k$ so that the periphery cycle $C_{2^{k+1}}$ encloses a finite face $F$. Then, place  a copy of $H^k$  inside the face $F$. Finally, add $2^{k+1}$ new vertices  $z_1,z_2,\cdots,z_{2^{k+1}}$  between the two periphery cycles $C_{2^{k+1}}$, and connect them to the four surrounding vertices, see the shaded ring on the right side of Figure~\ref{fig3}.

Clearly, $H^k \circ H^k$  is a plane graph. 
Let $XH^k$ be the $1$-plane graph obtained from $H^k \circ H^k$ by applying a $T_\times$-triangulation operation to each
quadrangle  of  $H^k \circ H^k$. 
Let $YH^k$ be the $1$-graph obtained from $H^k \circ H^k$ by applying  $T_\times$-triangulation operation to 
each quadrangle and  $K_1$-triangulation operation to each triangle of  $H^k \circ H^k$. 
The instance for $k=1$ can be seen 
in Figure~\ref{fig31}.

\begin{figure}[ht]
\centering
\includegraphics[width=0.9\textwidth]{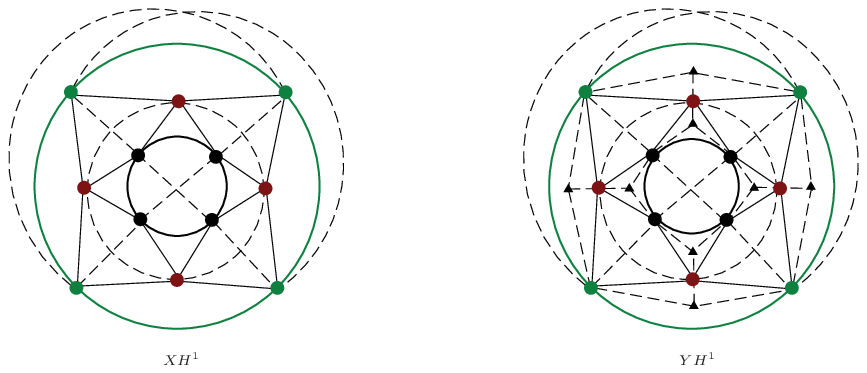}
\caption{The 1-plane graphs $XH^1$ and $YH^1$}
\label{fig31} 
\end{figure}

\begin{lemma}\label{cross-H}
For any $k\ge 1$, 
$|V(H^k)|=2^{k+2}-4$,
$|V(XH^k)|=5\cdot 2^{k+1}-8$,
\begin{equation*}
|V(YH^k)|=9\cdot 2^{k+1}-16
\quad \mbox{and}\quad 
cr(XH^k)=cr(YH^k)=3\cdot 2^{k+1}-6.
\end{equation*}
\end{lemma} 

\begin{proof}
By definition,
\begin{equation*}
|V(H^k)|=2^2+\cdots+2^{k+1}
=2^{k+2}-4
\end{equation*}
and
\begin{equation*}
|V(XH^k)|=|V(H^k\circ H^k)|=2|V(H^k)|+2^{k+1}
=5\cdot 2^{k+1}-8.
\end{equation*}

	 It can be observed in $H^k\circ H^k$, there are only triangular and quadrangular faces, and faces of the same type are not adjacent to each other. Let $F_3$ and $F_4$ denote the numbers of triangular and quadrangular faces of $H^k\circ H^k$, respectively.  Then, $3F_3=4F_4$ and $3F_3+4F_4=2|E(H^k\circ H^k)|$. By  Euler's formula, it follows that 
	\begin{equation*}
		|V(H^k\circ H^k)|-\frac{1}{2}(3F_3+4F_4)
		+(F_3+F_4)=2.
	\end{equation*}
	This implies that  
	\begin{equation*}
		F_3=\frac{4}{5}(|V(H^k\circ H^k)|-2)
		\quad \mbox{and}\quad  F_4=\frac{3}{5}(|V(H^k\circ H^k)|-2).
	\end{equation*}
Thus, 
\begin{equation*}
|V(YH^k)|=|V(H^k\circ H^k)|+F_3
=9\cdot 2^{k +1} - 16
\end{equation*}
and, combined with Lemma~\ref{optimal}, 
\begin{equation*}
cr(XH^k)=cr_{\times}(XH^k)=cr(YH^k)=cr_{\times}(YH^k)=F_4=3\cdot 2^{k+1}-6.
\end{equation*}
The result follows.
\end{proof}

Let $\setx=\{5\cdot 2^{k+1}-8: k\ge 1\}$
and
$\sety=\{9\cdot 2^{k+1}-16: k\ge 1\}$
be two sets of positive integers.  

\begin{proposition}\label{proc3}
For any $n\in \sety$, 
there  exists a $3$-connected  maximal $1$-plane graph $G$ of order $n$ 
with $cr(G)=\frac{1}{3}(n-2)$.
\end{proposition}

\begin{proof}
By routine checking, it is easy to show that $YH^k$ is a maximal 1-plane graph for all $k\ge 1$.  
 It is well known that every triangulation  is 3-connected. 
As $YH^k$ contains a  triangulation as a spanning subgraph, $YH^k$ is 3-connected.

 For any 
$n=9\cdot 2^{k+1}-16\in \sety$, by Lemma~\ref{cross-H},
$YH^k$ is of order $n$ and we have 
\begin{equation*}
cr(YH^k)=3\cdot 2^{k+1}-6
=\frac 13(n-2).
\end{equation*}
Thus, the result holds.
\end{proof}

\begin{proposition}\label{proc6}
For any $n\in \setx$, 
there  exists a $6$-connected  maximal 1-plane graph $G$ of order $n$ 
with $cr(G)=\frac{3}{5}(n-2)$.
\end{proposition}

\begin{proof}
By routine checking, it is easy to show that $XH^k$ is a maximal 1-plane graph for all $k\ge 1$.  
Next, we prove that the following claim.

\setcounter{claim}{0}
\begin{claim}\label{cl-5con}
$XH^k$ is $6$-connected.
\end{claim}
\begin{proof}  It is not difficult to find that $XH^k$ contains a spanning triangulation $P(XH^k)$ which is  (5,6)-regular, see the left side of Figure~\ref{fig32} for $k=3$. By Lemma~\ref{5connectivity}, $P(XH^k)$ is 5-connected, and thus $XH^k$ is as well. 

Suppose to the contrary that $XH^k$ is not 6-connected. Let $S$ be a minimum cut-set of $XH^k$. Then $|S|=5$. Clearly, $S$ is also a minimal cut-set of $P(XH^k)$. By Lemma~\ref{connectivity}, $S$ induces a separating cycle in $P(XH^k)$. By routine checking in $P(XH^k)$, one can observe that $S$ can only be the neighbors of the vertices of degree 5 (marked by $\bigstar$ in Figure~\ref{fig32}) in the kernel and periphery cycles. But we can verify that $XH^k-S$ is connected.  This contradicts the assumption that $S$ is a cut-set of $XH^k$. Consequently,  $XH^k$ is 6-connected.

Hence the claim holds.
\end{proof}

For any 
$n=5\cdot 2^{k+1}-8\in \setx$,  $XH^k$ is of order $n$ by 
Lemma~\ref{cross-H} and $6$-connected by Claim~\ref{cl-5con}.
It follows from Lemma~\ref{cross-H} again that
\begin{equation*}
cr(XH^k)=3\cdot 2^{k+1}-6
=\frac 35(n-2).
\end{equation*}

Thus, the result holds.
\end{proof}

\begin{figure}[ht]
\centering
\includegraphics[width=0.9\textwidth]{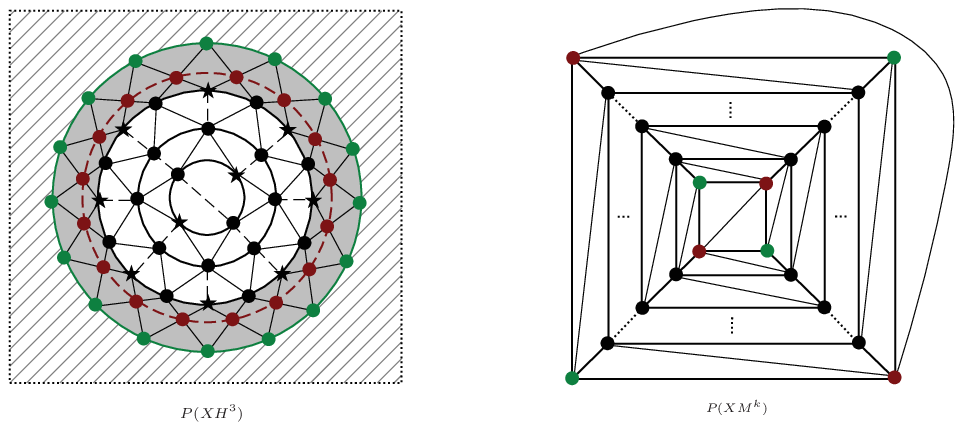}
\caption{The maximal plane graphs $P(XH^3)$ and $P(XM^k)$}
\label{fig32}
\end{figure}

Let $M^k=C_4\Box P_k$ be the Cartesian product of $C_4$ and $P_k$,
which are the cycle graph of order $4$
and the path graph of order $k$,
respectively. 
The left side of Figure~\ref{fig4} shows a planar drawing of $M^k$
in which each face is a quadrangle.  
Let $XM^k$ denote the graph obtained from $M^k$ by applying: 
\begin{itemize}
\setlength{\itemindent}{2em}
\item [{\rm(i)}]  a $K_2$-triangulation operation
	to the innermost quadrangle of $M^k$
	and adding a diagonal edge;
\item [{\rm(ii)}]  a $T_\times$-triangulation 
	operation
	to the outermost quadrangle of $M^k$; and 
\item [{\rm(iii)}]  $K_1$-triangulation operation
	to 
	each intermediate quadrangle and adding a diagonal edge.
\end{itemize}
The right side of Figure~\ref{fig4} illustrates this construction.  

\begin{figure}[htp]
	\centering
	\includegraphics[width=0.9\textwidth]{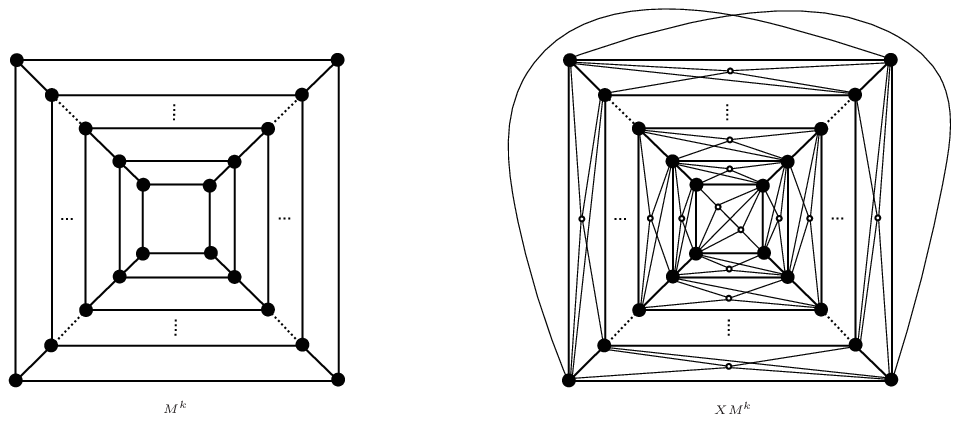}
	\caption{The plane graph $M^k=C_4\Box P_k$ and 1-plane graph $XM^k$}
	\label{fig4} 
\end{figure}

By the definitions of $M^k$ and $XM^k$, combined with Lemma \ref{optimal}, the next result follows directly.

\begin{lemma}\label{XM}
For any $k\ge 1$,  it follows that $|V(M^k)| =4k$, $|V(XM^k)|=8k-2$, and $cr(XM^k)=4k-2$.
\end{lemma} 

\def \setm {{\mathscr M}}

Let $\setm=\{8k-2: k\ge 1\}$ be a set of positive integers.

\begin{proposition}\label{pro2}
	For any $n\in \setm$, 
	there  exists a 4-connected maximal $1$-plane graph $G$ of order $n$ 
	with $cr(G)=\frac{1}{2}(n-2)$. 
\end{proposition}

\begin{proof}
Let $n=8k-2\in \setm$ for some $k\ge 1$.  By routine checking, it is easy to show that $XM^k$ is a maximal 1-plane graph.  
By Lemma~\ref{XM},  $XM^k$ is  of order $n$, and
\begin{equation*}
cr(XM^k)=4k-2=\frac 12 (n-2).
\end{equation*}

Now it remains to show that $XM^k$
is $4$-connected.
It is routine to check that 
$XM^k$ is $4$-connected if $k=1$. 
Now assume that $k\ge 2$.
Observe that $XM^k$ contains a triangulation $P(XM^k)$ (illustrated in the right side of Figure~\ref{fig32}) 
with degree sequence:
\begin{equation*}
\underbrace{6,6,\ldots,6}_{4k-8},5,5,5,5,4,4,4,4,
\end{equation*}
where the four vertices of degree $4$
are marked in green
and the four vertices of degree 5
are marked in red
(see the right side of Figure~\ref{fig32}). 
Then, by Lemma~\ref{4connectivity},
$P(XM^k)$ is 4-connected. 

By comparing graphs $XM^k$ and $P(XM^k)$, we find that  $XM^k$ can be constructed from  $P(XM^k)$ by continuously adding $4(k-1)$ vertices of degree 4;  adding a $K_2$ and six edges joining the two vertices $v_1$ and $v_2$ in this $K_2$ to vertices in $P(XM^k)$ so that each $v_i$ is of degree $4$; 
and finally
adding one edge connecting  two vertices of degree 4.
Obviously, each expansion preserves 4-connectivity, thereby confirming that $XM^k$ is also 4-connected.
\end{proof}

\begin{proposition}\label{pro3}
	For any $n\in \{24, 56\}$, 
there exists a 7-connected maximal $1$-plane graph $G$ of order $n$
with $cr(G)=\frac{3}{4}n$.
\end{proposition}
\begin{proof}

Let $T_1$ and $T_2$ be the two 
$1$-plane graphs shown in Figure~\ref{fig5}, respectively. 
Observe that $|V(T_1)|=24$,  $|V(T_2)|=56$,  and, combined with Lemma \ref{optimal},
\begin{equation*}
cr(T_1)=cr_{\times}(T_1)=18=\frac{3}{4}\times 24,
\quad 
cr(T_2)=cr_{\times}(T_2)=42=\frac{3}{4}\times 56.
\end{equation*}
It can be verified directly that 
 for $i=1,2$,  
$T_i$ is a 7-connected maximal $1$-plane graph.
Thus, the result holds.
\end{proof}

\begin{remark}\label{remark7-con}
We have not found infinitely many $7$-connected  maximal 1-plane graphs $G$ with the property that 
there are $\frac{3}{4}|V(G)|$
crossings in $G$. 
It is quite possible that there are only a finite number of such classes of graphs.
\end{remark}

\begin{figure}[htp]
\centering
\includegraphics[width=0.9\textwidth]{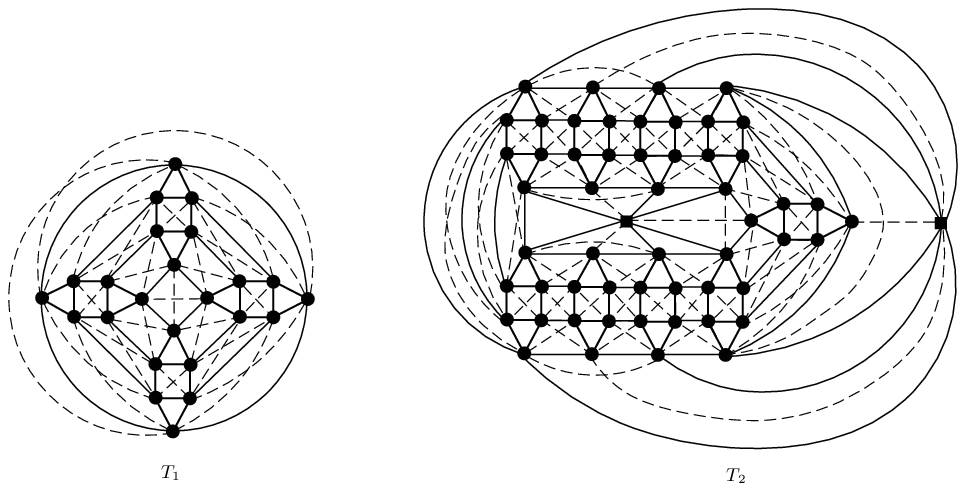}
\caption{Two 1-plane graphs with connectivity 7}
\label{fig5}
\end{figure}

\section{Prove the main results}
\label{sec-main}

In this section, we complete the proofs of Theorems~\ref{main1} and~\ref{main2}.

\vspace{0.4 cm}

\noindent  {\bf Proof	
of Theorem~\ref{main1}}: 
Let $G\in \setg_k$. 
By Lemma~\ref{lemtr-3} and~\ref{lemtr-k}, we know that $G^\times$ is a triangulation. 

First, consider the cases for $k=3,4,5$. For the convenience of counting, 
we construct an auxiliary bipartite graph $H$ with 
a bipartition 
$(V_{bl}(G_P^*), V_{rd}(G_P^*))$, where a vertex $x\in V_{bl}(G_P^*)$
is adjacent to a vertex  $y\in V_{rd}(G_P^*)$ 
if and only if
$x$ and $y$ correspond to two adjacent vertices in $G_P^*$.

By Observation~\ref{obs} and~\ref{tri} (ii),   we know that each red vertex is adjacent to at least one red vertex in $G_P^*$ and $G_P^*$ is 3-regular.  Therefore, it follows that $deg_H(v)\le 2$ for all $v\in V_{rd}(G_P^*)$.  
It follows that
\begin{equation}\label{EQ1}
|E(H)|\le 2|V_{rd}(G_P^*)|.
\end{equation}
By Corollary~\ref{adjacent-vertex},  each blue vertex in $G_P^*$  is adjacent to at most  $5-k$ blue vertices  for $3\le k\le 5$,
implying that $deg_H(v)\ge 3-(5-k)=k-2$ for each $v\in V_{bl}(G_P^*)$. 
Thus, we have
\begin{equation}\label{EQ2}
|E(H)|\ge (k-2)|V_{bl}(G_P^*)|.
\end{equation}
Combining inequality (\ref{EQ1}) and  (\ref{EQ2}), we obtain 
\begin{equation}\label{EQ3}
 (k-2)|V_{bl}(G_P^*)|\le 2|V_{rd}(G_P^*)|.
 \end{equation}
From Observation~\ref{tri} (ii)  it follows that
\begin{equation}\label{EQ4}
|V_{bl}(G_P^*)|+|V_{rd}(G_P^*)|=|V(G_P^*)|=2n-4.
 \end{equation}
Therefore, combining inequality (\ref{EQ3}) with  equality (\ref{EQ4}) we deduce that 
\begin{equation}\label{EQ5}
|V_{rd}(G_P^*)|\ge \frac{(k-2)(2n-4)}{k}.
 \end{equation}
By Observation~\ref{tri} (ii) and Remark~\ref{1pcross}, it follows from inequality (\ref{EQ5}) that 
\begin{equation*}
cr(G)= \frac 12 |V_{rd}(G_P^*)|
\ge \frac{(k-2)(n-2)}{k}.
 \end{equation*}
This shows that the conclusion holds for $k=3,4,5$.

When $k=6$, the conclusion follows directly because, if the lower bound on $cr(G)$ holds for $k=5$ , then it also holds for $k=6$. Hence, we only need to consider the case for $k=7$.

In this case,  $deg_G(v)\ge 7$ for all $v\in V(G)$. By Lemma~\ref{trcredge},
each vertex in $G$ is incident with at least three crossing edges (corresponding to three crossing points). 
 Meanwhile, observe that each crossing involves four vertices of $G$. 
 Let $\Psi$ be the set of ordered pairs $(c , v)$, where $c$ is a crossing point in $G$ and $v$ is a vertex in $G$. Thus, 
\begin{equation*}
3n\le |\Psi|\le 4cr_{\times}(G).
\end{equation*}
 This implies, combined with Lemma~\ref{optimal}, that 
\begin{equation*}
cr(G)=cr_{\times}(G)\ge \frac{3n}{4},
\end{equation*}
 as desired. 

When $k\in \{3,4,6,7\}$, 
the tightness of the lower bound
for $cr(G)$ follows directly from 
Propositions~\ref{proc3}, \ref{proc6}, \ref{pro2} and \ref{pro3}.
Thus, Theorem~\ref{main1} holds.
\hfill 
$\square$

\vspace{0.4 cm}

We now proceed to prove Theorem~\ref{main2}.
\vspace{0.4 cm}

\noindent 
{\bf Proof of Theorem~\ref{main2}}:
Since $G\in \mathcal{G}_k$,   by Lemma~\ref{lemtr-3} and~\ref{lemtr-k},
$G^\times$ is a triangulation. 
Hence, by Remark~\ref{1pcross}, 
$|E(G)|=3n-6+cr(G)$. 
Thus, the lower bound of $|E(G)|$ in  Theorem~\ref{main2} follows directly from that in Theorem~\ref{main1}. 

As $|E(G)|=3n-6+cr(G)$,
the tightness of the lower bounds
for $|E(G)|$ follows from 
Theorem~\ref{main1}. 
\hfill 
$\square$

\vspace{0.4 cm}

A 1-plane graph
$G$ is called {\it near optimal} (as introduced by Suzuki in \cite{YS})  if it satisfies the following conditions:
\begin{enumerate}
\setlength{\itemindent}{2em}
	\item[(i)]  every face in the subgraph $H$ 
	induced by all non-crossing edges
	of $G$
	is either triangular or quadrangular, 
	\item[(ii)] any quadrangular face
	of $H$ bounded by $v_0v_1v_2v_3$  contains the unique crossing point
created by the pair of crossing edges $v_0v_2$ and $v_1v_3$, 
	and 
\item[(iii)] no edge in $G$ is shared by two distinct triangular faces.
\end{enumerate} 

\begin{remark}\label{planargraph}
Y. Suzuki \cite{YS} proved that every near optimal 1-plane graph with $n$ vertices has at least $\frac{18}{5}(n-2)$ edges and every 5-connected maximal 1-plane graph $G$ is near optimal. This result directly implies that every 5-connected maximal 1-plane graph has at least $\frac{18}{5}(n-2)$ edges, and this lower bound is consistent with that of Theorem~\ref{main2} when $k=5$. 
\end{remark}

By the definition of a maximal 1-planar graph,  the following conclusion  is obvious.

\begin{corollary}\label{crossing1planar}
The lower bounds in Theorems~\ref{main1} and \ref{main2} also hold 
	if the condition that $G$ is a  
$k$-connected maximal $1$-plane graph
is replaced by that $G$ is a $k$-connected maximal 1-planar graph,
where $3\le k\le 7$.
\end{corollary}

\section{Conclusion and open problems}
\label{sec:problem}

Recall that $XH^k$ and $YH^k$ constructed in Section~\ref{sec-exmg} are maximal 1-plane graphs for all $k\ge 1$. 
Generally speaking, determining whether a 1-plane graph is maximal is relatively straightforward. However, verifying the maximality of its underlying 1-planar graph poses significant challenges. Through a tedious verification, we can prove that $XH^1$ and $YH^1$ are maximal 1-plane graphs. Furthermore, we believe that the following is true.

\begin{conjecture}\label{conj2}
$XH^k$ and $YH^k$ are maximal 1-planar graphs for all $k\ge 1$.
\end{conjecture}

Let $\mathcal{G}$ be a family of graphs, and for any positive integer 
$n$, let $m(\mathcal{G}, n)$ 
denote the minimum
value of $|E(G)|$ over all maximal  graphs $G\in \setg$ 
with $|V(G)|=n$.
In \cite{DTY1},  Hud\'{a}k,  Madaras and Suzuki proved that for each rational number $\frac{p}{q}$ in the interval $[\frac{8}{3},4]$,  there exist infinitely many integers $n$ with the existence of  a 2-connected maximal 1-planar graph with $n$ vertices and $\frac{p}{q}(n-2)$ edges. Furthermore, they proposed the following conjecture.

\begin{conjecture}[\cite{DTY1}]\label{conj}
For the family $\overline{\mathcal{P}^\star}$ of 3-connected maximal 1-planar graphs,  $m(\overline{\mathcal{P}^\star},n)=\frac{18}{5}n+c$, where $c$ is a constant.
\end{conjecture}

By Corollary~\ref{crossing1planar}, the size of a 3-connected maximal 1-planar graph is at least $\frac{10}{3}n-\frac{20}{3}$.  
If Conjecture~\ref{conj2} holds, it indicates that there exist arbitrarily large 3-connected maximal 1-planar graphs that attain this lower bound of $\frac{10}{3}n-\frac{20}{3}$, thereby disproving Conjecture~\ref{conj}.

In this article, we obtain some 
partial results on 
the minimum crossing number
 and the minimum size
among all $k$-connected 
maximal $1$-plane graphs
of order $n$, where $3\le k\le 7$.
However,  the following problems 
on the minimum crossing number
and the minimum size of 
maximal $1$-plane graphs
remain open.

\begin{problem}\label{prob1} 
What is the minimum value of $cr(G)$ over all maximal $1$-plane graphs $G$ of order $n$ with connectivity 2?
\end{problem}

\begin{problem}\label{prob2} 
What is the minimum value of $cr(G)$ and $|E(G)|$
over all maximal $1$-plane graphs $G$ of order $n$ with connectivity 3?
\end{problem}

\begin{problem}\label{prob3} 
Does there exist a maximal $1$-plane graph $G$ of order $n$ with connectivity 5 such that $cr(G)=\frac{3}{5}(n-2)$ or  $|E(G)|=\frac{18}{5}(n-2)$?
\end{problem}

\end{document}